\documentclass[11pt,letterpaper]{amsart}

\usepackage{tikz}
\usetikzlibrary{arrows.meta}
\usepackage{url}
\usepackage{graphicx}
\urldef\myurl\url{https://en.wikipedia.org/wiki/Ramanujan%E2%80%93Sato_series}

\usepackage{amssymb}
\usepackage[T1]{fontenc}
\usepackage{mathscinet}
\usepackage{enumitem}

\newcommand{\CT}{\operatorname{CT}}

\allowdisplaybreaks%

{\theoremstyle{plain}%
 \newtheorem{theorem}{Theorem}
 
 \newtheorem{lemma}{Lemma}% 
}
{\theoremstyle{remark}
\newtheorem{remark}{Remark}
}
{\theoremstyle{definition}

}

\title[Ramanujan--Sato series of level 11]{On a problem on Ramanujan--Sato series of level 11}
\author{John M.\ Campbell}
\address{Department of Mathematics and Statistics, Dalhousie University,
Halifax, NS B3H 4R2, Canada}
\email{jh241966@dal.ca}
 


\begin{document}

\begin{abstract}
 A longstanding open problem concerning Ramanujan--Sato series is given by determining an expression for the coefficients $c_{11}(n)$ 
 associated with Ramanujan--Sato series of level 11 with a fixed, nested, finite binomial sum for all nonnegative integers $n$, by direct 
 analogy with the known     formulas for the coefficients for lower-level Ramanujan--Sato series (and this is clarified and made explicit in our  
  paper).   This open problem was given by Cooper et al.\ in 2015 [\emph{J.\ Math.\ Anal.\ Appl.}], 
 in Cooper's 2017 textbook on Ramanujan's theta functions, 
 and in Cooper's 2025 contribution to the \emph{Contemp.\ Math.} series. 
 This problem seems to have remained open. 
 We succeed in solving this problem, by introducing and proving a double binomial 
 sum formula for $c_{11}(n)$ that is directly analogous to the coefficient formulas for lower-level Ramanujan--Sato series. 
 \end{abstract}

\maketitle

\section{Introduction and background}
 Ramanujan's series for $\frac{1}{\pi}$ \cite[pp.\ 352--354]{Berndt1994} \cite{Ramanujan1914} are among Ramanujan's most celebrated 
 discoveries. A major area of modern-day research related to Ramanujan's series for $\frac{1}{\pi}$ is given by 
 Ramanujan--Sato series for $\frac{1}{\pi}$. 
 In this paper, we solve, by building on a construction described by GPT-5.6 Sol Pro (which, in turn, was based 
 on the author's suggestions and guidance), 
 a longstanding open problem 
 on Ramanujan--Sato series that dates back to 
 a 2015 paper by Cooper et al.\ \cite{CooperGeYe2015}
 and that has long been highlighted as an open problem in the Wikipedia page 
 on Ramanujan--Sato series\footnote{See \myurl.}. 
 This open problem was also given by Cooper in 2025 \cite{Cooper2025}
 and was given in Cooper's 2017 text 
 on Ramanujan's theta functions \cite[p.\ 567]{Cooper2017}. 

 Set \(c_j=0\) for \(j<0\), and let \(c_0=1\). Define
 $(c_k : k \in \mathbb{N}_{0})$ via the D-finite recurrence such that 
\begin{align}
\begin{split}
 (k+1)^{3} c_{k+1} & = 
 2(2k+1)(5k^2 + 5 k + 2) c_{k} 
 - \\ 
 & \quad \quad 8 k (7k^2 + 1) c_{k-1} + 22 k(k-1)(2k-1) c_{k-2}, 
\end{split}\label{splitrec}
\end{align}
 giving rise to the initial terms 
\begin{multline*}
 (c_k : k \in \mathbb{N}_{0}) = \\ 
 \big( 1, 4, 28, 268, 3004, 36784, 476476, 6418192, 88986172, 1261473136, \ldots \big), 
\end{multline*}
 as indexed in the On-line Encyclopedia of Integer Sequences\footnote{See \url{https://oeis.org/}.} as {\tt A284756}. 
 Now, let $q = e^{2 \pi i \tau}$ for $\Im(\tau) > 0$, 
 and let the Dedekind eta function be defined so that 
 $$ \eta(\tau) = q^{1/24} \prod_{m=1}^{\infty}(1 - q^{m}). $$
 Now, define 
\begin{align*}
 a_{11}(q) & = \sum_{r, s \in \mathbb{Z}} q^{r^2 + r s + 3 s^2}, \\
 z_{11}(q) & = a_{11}^{2}(q), \ \text{and} \\ 
 x_{11}(q) & = \frac{\eta^{2}(\tau) \eta^{2}(11 \tau) }{z_{11}(q)}.
\end{align*}
 By Cooper, Ge, and Ye \cite[Thm.~4.7]{CooperGeYe2015}, the   sequence $(c_k : k \in \mathbb{N}_{0})$ defined above    occurs in the   
  normalized modular expansion
\begin{equation}\label{normalized}
 z_{11}(q)=\sum_{n=0}^{\infty}c_nx_{11}^{n}(q), 
\end{equation}
 so we write \(c_{11}(n):=c_n\). Suitable modular specializations of the weighted generating series
 $$ \sum_{n=0}^{\infty} c_{11}(n) (n + \lambda) X^{n} $$
 then yield Ramanujan--Sato series for $\frac{1}{\pi}$
 of level $11$. For brevity, we often write $c_n$ in place of $c_{11}(n)$. 

\begin{remark}
 With regard to the expansion in \eqref{normalized}, since \(x_{11}(q)=q+O(q^2)\), this expansion determines the 
 sequence $(c_n : n \in \mathbb{N}_{0})$  uniquely.
 The uniqueness asserted here is relative to     the pair       $(a_{11}(q),   x_{11}(q))$;  
      it is not a claim
that the  level \(11\), independently of all choices of
modular groups, Hauptmoduls,  and   modular forms,        determines
a unique coefficient sequence. This contrasts with Cooper's
level-\(10\) theory, in which the three distinct modular pairs
\[
(X_A,Z_A),\qquad (X_B,Z_B),\qquad (X_C,Z_C)
\]
give rise to the three coefficient sequences \(T_A(n)\), \(T_B(n)\),
and \(T_C(n)\)
\cite[Chap.~10, Ex.~9]{Cooper2017}.
\end{remark}

 Notable instances of Ramanujan--Sato series of level 11 include \cite[Eqs.\ (6.3)--(6.4)]{CooperGeYe2015} $$ \frac{1}{2\pi} = 
 \frac{7 \sqrt{11}}{121} \sum_{n=0}^{\infty} c_{n} \left( n + \frac{3}{14} \right) \frac{1}{22^{n}} $$ and $$ \frac{1}{\pi} = 
 \frac{221 \sqrt{11}}{484} \sum_{n=0}^{\infty} c_{n} \left( n + \frac{67}{221} \right) \left( -\frac{1}{44} \right)^{n}. $$

 A longstanding open problem concerning the integer sequence
 $(c_{n} : n \in \mathbb{N}_{0})$ may be formulated as follows. 
 How can $c_{n}$ be written as a fixed, nested, finite binomial sum 
 for all $n \in \mathbb{N}_{0}$, by direct analogy
 with the known formulas for the coefficients for the standard classified lower-level Ramanujan--Sato series? 
 To make this problem explicit and precise, 
 the known formulas for the coefficients of the standard classified Ramanujan--Sato series of levels $\ell \leq 10$
 are all of the form 
\begin{equation}\label{makeexplicit}
t(n) =
\sum_{k_1}\sum_{k_2}\cdots\sum_{k_d}
 z^{p(n, k_1, k_2, \ldots, k_d)} H(n,k_1, k_2, \ldots,k_d) 
\end{equation}
 and satisfy the following: The number \(d\) of summation indices is fixed (independently of \(n\)), allowing for $d = 0$; the finite 
 summation region
is specified by a fixed system of affine-linear inequalities in \(n\)
and the summation indices; the summand factor \(H\) is a fixed 
 product of fixed integer powers of binomial coefficients whose
arguments are affine-linear forms in \(n,k_1, k_2, \ldots,k_d\); 
 the expression $z \neq 0$ is an integer and $p(n, k_1, k_2, \ldots, k_d)$ is a fixed and integer-valued
 affine-linear form
 in the given arguments. 
 Observe that neither
 the summand nor the summation region is permitted to depend on the
 values of the sequence $(t(n) : n \in \mathbb{N}_{0})$ itself. The main problem addressed
 in this paper is to find a formula for $c_{11}(n)$ 
 satisfying the given conditions associated with 
 \eqref{makeexplicit}. 

 For example, by defining
\begin{equation}\label{beta2binomial}
 u_{10B}(n) = \sum_{a=0}^{n} \sum_{b=0}^{a} \binom{2n}{n} \binom{2a}{a}^{-1} \binom{n}{a} \binom{a}{b}^{4}, 
\end{equation}
 the formula $$ \frac{1}{\pi} = \frac{1}{17 \sqrt{95}} \sum_{n=0}^{\infty} u_{10B}(n) \frac{19 \cdot 1824 n + 3983}{ (76^2 + 4)^{n + 
 \frac{1}{2}}} $$ is an instance of a Ramanujan--Sato series of level 10, 
 with reference to the work of Cooper on such series \cite{Cooper2012}. 
 Observe that the coefficient formula in \eqref{beta2binomial} satisfies
 all of the conditions associated with \eqref{makeexplicit}. 
 This illustrates the interest in deriving an analogue of coefficient formulas as in \eqref{beta2binomial}, 
 for the level-11 case. 

 In this paper, we introduce and prove the formula 
\begin{equation}\label{s11Aexplicit}
 c_{11}(n) = \sum_{a=0}^{n} 
 \sum_{b=0}^{n-3a}
 \binom{n}{a} \binom{n-a}{a+b}^{2}
 \binom{2a+2b}{b}
 \binom{2n-4a-2b}{n-3a-b}, 
\end{equation}
 noting that all of the conditions associated with \eqref{makeexplicit} are satisfied, 
 thus solving the problem indicated above due to Cooper et al.\ \cite{CooperGeYe2015}. 

\begin{remark}
 We make use of the standard convention whereby empty sums vanish, so that 
 \eqref{s11Aexplicit} may be rewritten so that 
\begin{equation*}
c_{11}(n)
=
\sum_{\substack{a,b\geq0\\3a+b\leq n}}
\binom{n}{a}
\binom{n-a}{a+b}^{2}
\binom{2a+2b}{b}
\binom{2n-4a-2b}{n-3a-b}, 
\end{equation*}
 again noting that the conditions associated with \eqref{makeexplicit} are satisfied. 
\end{remark}

 In addition to how this is motivated by the given references from Cooper et al., 
 this is motivated by the work of Huber et al.\ that involves level-11 Ramanujan--Sato series 
 \cite{HuberSchultzYe2023}. 
 
 The key to our discovery of and proof of the explicit formula in \eqref{s11Aexplicit} 
 is given by a Laurent polynomial given by Fei in the context of the application of the Mahler 
 measure for 3D Landau--Ginzburg potentials \cite{Fei} 
 along with a corresponding determinantal differential equation of order 3 
 related to the work of Golyshev on pencils of certain $K3$ surfaces \cite{Golyshev2007}. 

\section{A Laurent polynomial given by Fei} 
 We work in the ring 
\[
\mathcal{R} = \mathbb{Z}[x^{\pm1},y^{\pm1},z^{\pm1}]
\]
 of Laurent polynomials in \(x,y,z\). Every
 element $F$ in this ring 
 has a unique finite expansion
\[
F(x,y,z)=
\sum_{(r,s,u)\in\mathbb Z^3}
\mathcal{C}_{r,s,u}x^r y^s z^u,
\]
where all but finitely many of the coefficients \(\mathcal{C}_{r,s,u}\) vanish.
We define the constant-term operator by
\[
\CT_{x,y,z}F:=
\mathcal{C}_{0,0,0}
=
[x^0y^0z^0]F.
\]
More generally, if
\[
\mathcal F(t)=\sum_{n\geq0}F_n(x,y,z)t^n
\in \mathcal{R}[[t]],
\]
then \(\CT_{x,y,z}\) is applied coefficientwise:
\[
\CT_{x,y,z}\mathcal F(t)
:=
\sum_{n\geq0}
\bigl(\CT_{x,y,z}F_n(x,y,z)\bigr)t^n.
\]

Let the Laurent polynomial \(\Lambda(x,y,z)\) be defined by
\begin{equation}
 \Lambda(x,y,z)=
 \frac{(x+y+xy+yz)(x+z+xz+yz)}{xyz}
 +\frac{1}{yz}.
 \label{eq:Lambda}
\end{equation}
This is the representative listed in the \(V_{22}\) row of
Fei's Table~1~\cite{Fei}.

For each \(n\geq0\), define
\begin{equation}
 b_n:=\CT_{x,y,z}\Lambda(x,y,z)^n.
 \label{eq:bn-ct}
\end{equation}
Thus \(b_n\) is, explicitly, the coefficient of the monomial
\(x^0y^0z^0\) in the finite Laurent polynomial \(\Lambda^n\).
The formal period of \(\Lambda\) is
\begin{equation}
 P(t):=\sum_{n\geq0}b_nt^n.
 \label{eq:period-series}
\end{equation}
Since
\[
\frac{1}{1-t\Lambda}
=
\sum_{n\geq0}t^n\Lambda^n
\]
as an identity in the formal power-series ring
\(\mathcal R[[t]]\), applying the constant-term operator
coefficientwise gives the equivalent representation
\begin{equation}
 P(t)
 =
 \CT_{x,y,z}\frac{1}{1-t\Lambda(x,y,z)}
 =
 \sum_{n\geq0}
 \bigl(\CT_{x,y,z}\Lambda(x,y,z)^n\bigr)t^n.
 \label{eq:period}
\end{equation}
 The constant-term interpretation in~\eqref{eq:period} is the standard
residue expansion of the quantum period; see Fei's Definition~1.1 and
Remark~1.2~\cite[Definition~1.1 and Remark~1.2]{Fei}.

\begin{lemma}\label{lem:recurrence}
 The period coefficients $b_n$ satisfy the recurrence \eqref{splitrec}, with $b_0=1$. Consequently,
\begin{equation}
 b_n=c_n\qquad(n\geq 0).
 \label{eq:periodequalssequence}
\end{equation}
\end{lemma}

\begin{proof}
Let $ \vartheta:=t\frac{d}{dt}$. 
Fei \cite[Eq.~(1.1)]{Fei} writes the relevant $D3$ operator in the form
\begin{multline*}
 \vartheta^3
 +t\Bigl(\vartheta+\tfrac12\Bigr)
 \bigl(\alpha_3(\vartheta^2+\vartheta)+\beta_1\bigr) 
 +t^2(\vartheta+1)
 \bigl(\alpha_2(\vartheta+1)^2+\beta_0\bigr) + \\
 \alpha_1t^3(\vartheta+2)
 \Bigl(\vartheta+\tfrac32\Bigr)(\vartheta+1)
 +\alpha_0t^4(\vartheta+3)(\vartheta+2)(\vartheta+1).
\end{multline*}
 The heading of the last column of Fei's
Table~1 is
$-(\alpha_3,\alpha_2,\alpha_1;\beta_1,\beta_0)$, and the $V_{22}$ row is
\[
 (20,-56,44;8,-8),
\]
with $\alpha_0=0$ \cite{Fei}. Hence the operator
annihilating $P(t)$ is
\begin{align}
\begin{split}
 \mathcal L_{11}:={} 
 & \vartheta^3
 -t\Bigl(\vartheta+\tfrac12\Bigr)
 \bigl(20(\vartheta^2+\vartheta)+8\bigr) 
 +t^2(\vartheta+1)
 \bigl(56(\vartheta+1)^2+8\bigr) \\
 & \quad -44t^3(\vartheta+2)
 \Bigl(\vartheta+\tfrac32\Bigr)(\vartheta+1), 
\end{split}\label{eq:D3operator}
\end{align}
 so that $\mathcal L_{11}P=0$.

 Here and below, differential operators act on expressions to their right; thus
\[
t^jQ(\vartheta)F(t):=t^j\bigl(Q(\vartheta)F(t)\bigr).
\]
In particular, the displayed powers of \(t\) are not commuted through
the factors involving \(\vartheta\). 
 Taking the coefficient of \(t^m\) in
\(\mathcal L_{11}P(t)=0\), and setting \(b_j=0\) for \(j<0\), gives 
\begin{align}
\begin{split}
 0={}&m^3b_m
 -2(2m-1)(5m^2-5m+2)b_{m-1} 
 + 8 (m- \\
 & 1)\bigl(7(m-1)^2+1\bigr)b_{m-2} - 22(m-1)(m-2)(2m-3)b_{m-3}, 
\end{split}\label{eq:bmrecurrence}
\end{align}
 and this can be verified, for example, 
 by writing 
\[
 \Bigl(m-\tfrac12\Bigr)\bigl(20m(m-1)+8\bigr)
 =2(2m-1)(5m^2-5m+2),
\]
 and by rewriting the other two shifted factors similarly.

Now, put $m=n+1$ in~\eqref{eq:bmrecurrence} and move all terms except
$(n+1)^3b_{n+1}$ to the right. The result is exactly
\eqref{splitrec} with $b$ in place of $c$. Also,
\[
 b_0=\CT_{x,y,z}1=1.
\]
Since the coefficient $(n+1)^3$ never vanishes for $n\geq 0$, the
recurrence and the initial value determine the sequence uniquely.
Therefore $b_n=c_n$ for all $n\geq 0$.
\end{proof}

\begin{remark}
Fei's Table~2 associates the $V_{22}$ family with
$\Gamma_0^+(11)$~\cite[Table~2]{Fei}. This is consistent with the
level-$11$ interpretation. The proof above does not need a separate
modular-function identification: equality of the normalized period
coefficients follows already from the identical recurrence and initial
value.
\end{remark}

\section{Explicit extraction of the constant term}

Set
\begin{equation}
 A:=x+y+xy+yz=x(1+y)+y(1+z)
 \label{eq:Adef}
\end{equation}
and
\begin{equation}
 B:=x+z+xz+yz=x(1+z)+z(1+y).
 \label{eq:Bdef}
\end{equation}
Then
\[
 \Lambda=\frac{AB}{xyz}+\frac{1}{yz}.
\]

\begin{lemma}\label{lem:firstsum}
For every $n\geq 0$,
\begin{equation}
 b_n=
 \sum_{r=0}^{n}\binom{n}{r}
 \sum_{j=0}^{r}
 \binom{r}{j}^{2}
 \binom{2j}{n-r+j}
 \binom{2r-2j}{n-j}.
 \label{eq:rjsum}
\end{equation}
\end{lemma}

\begin{proof}
By the binomial theorem,
\begin{align}
 b_n
 &=\CT_{x,y,z}
 \sum_{r=0}^{n}\binom{n}{r}
 \left(\frac{AB}{xyz}\right)^r
 \left(\frac{1}{yz}\right)^{n-r}\notag\\
 &=\sum_{r=0}^{n}\binom{n}{r}
 [x^ry^nz^n]A^rB^r.
 \label{eq:ct1}
\end{align}
From~\eqref{eq:Adef} and~\eqref{eq:Bdef},
\begin{align}
 A^r
 &=\sum_{j=0}^{r}\binom{r}{j}
 x^jy^{r-j}(1+y)^j(1+z)^{r-j},
 \label{eq:Ar}\\
 B^r
 &=\sum_{\ell=0}^{r}\binom{r}{\ell}
 x^\ell z^{r-\ell}(1+z)^\ell(1+y)^{r-\ell}.
 \label{eq:Br}
\end{align}
In the product $A^rB^r$, the extraction of $x^r$ forces
$j+\ell=r$, hence $\ell=r-j$. Therefore
\begin{align*}
 [x^ry^nz^n]A^rB^r
 &=\sum_{j=0}^{r}\binom{r}{j}^{2}
 [y^{n-r+j}](1+y)^{2j}
 [z^{n-j}](1+z)^{2r-2j}\\
 &=\sum_{j=0}^{r}\binom{r}{j}^{2}
 \binom{2j}{n-r+j}
 \binom{2r-2j}{n-j}.
\end{align*}
Substitution in~\eqref{eq:ct1} proves~\eqref{eq:rjsum}.
\end{proof}

 This leads us to our main result. 

\begin{theorem}
 The binomial sum formula for $c_{11}(n)$ in \eqref{s11Aexplicit}
 holds for all nonnegative integers.
\end{theorem}

\begin{proof}
By Lemma~\ref{lem:recurrence}, $c_n=b_n$. Starting from
\eqref{eq:rjsum}, make the change of variables
\begin{equation}
 a:=n-r,\qquad b:=j-a.
 \label{eq:change}
\end{equation}
Thus
\[
 r=n-a,\qquad j=a+b.
\]
We now determine the exact nonzero range.

The factor
\[
 \binom{2j}{n-r+j}=\binom{2j}{a+j}
\]
is nonzero only when $a\leq j$, which is equivalent to $b\geq 0$.
The factor
\[
 \binom{2r-2j}{n-j}
\]
is nonzero only when
\[
 n-j\leq 2r-2j.
\]
Using $r=n-a$ and $j=a+b$, this inequality is equivalent to
\[
 b\leq n-3a.
\]
Thus the nonzero terms are exactly those with
$$
 0\leq a\leq \left\lfloor\frac n3\right\rfloor,
 \qquad 0\leq b\leq n-3a.
$$
On this range,
\begin{align*}
 \binom{n}{r}&=\binom{n}{a},\\
 \binom{r}{j}&=\binom{n-a}{a+b},\\
 \binom{2j}{n-r+j}
 &=\binom{2a+2b}{2a+b}
 =\binom{2a+2b}{b},\\
 \binom{2r-2j}{n-j}
 &=\binom{2n-4a-2b}{n-a-b}\\
 &=\binom{2n-4a-2b}{n-3a-b},
\end{align*}
where the last equality uses symmetry of binomial coefficients.
 Substituting these identities into~\eqref{eq:rjsum} gives
\eqref{s11Aexplicit}, completing the proof.
\end{proof}

 The only non-elementary input in the proof is the published Picard--Fuchs operator for the $V_{22}$ Laurent period. Once 
 Lemma~\ref{lem:recurrence} is established from that operator, the
constant-term extraction above is purely algebraic. Conversely,
one may use~\eqref{s11Aexplicit} as an explicit candidate and verify
\eqref{splitrec} independently by multivariate creative
telescoping; such a verification is supplementary rather than necessary
for the proof given here.

\subsection*{Statements and Declarations}

\noindent\textbf{Funding.}
No funding was received for conducting this study.

\medskip

\noindent\textbf{Competing interests.}
The author has no relevant financial or non-financial interests to
disclose.

\medskip

\noindent\textbf{Data availability.}
No datasets were generated or analyzed during the current study.
The results of the article are established by theoretical arguments.

\subsection*{Acknowledgements}
 The author acknowledges extensive interactions with GPT-5.6 Sol Pro during the exploratory and proof-development stages of this work. All 
 AI-generated suggestions were substantially revised, corrected, and independently verified by the author, who assumes full responsibility 
 for the mathematical content. The author's use of AI is fully in accordance with editorial standards on the responsible, ethical, and 
 transparent use of AI. 
 
\bibliographystyle{plain}
\bibliography{seprefe}

@article{Ramanujan1914,
 author = {Ramanujan, S.},
 title = {Modular equations and approximations to {{\(\pi\)}}.},
 fjournal = {The Quarterly Journal of Pure and Applied Mathematics},
 journal = {Quart. J.},
 volume = {45},
 pages = {350--372},
 year = {1914},
 language = {English},
 url = {www.imsc.res.in/~rao/ramanujan/CamUnivCpapers/Cpaper6/page1.htm},
 zbMATH = {2619245},
 JFM = {45.1249.01}
}

@incollection {Golyshev2007,
    AUTHOR = {Golyshev, Vasily V.},
     TITLE = {Classification problems and mirror duality},
 BOOKTITLE = {Surveys in geometry and number theory: reports on contemporary
              {R}ussian mathematics},
    SERIES = {London Math. Soc. Lecture Note Ser.},
    VOLUME = {338},
     PAGES = {88--121},
 PUBLISHER = {Cambridge Univ. Press, Cambridge},
      YEAR = {2007},
      ISBN = {978-0-521-69182-6; 0-521-69182-6},
   MRCLASS = {14J45},
  MRNUMBER = {2306141},
MRREVIEWER = {Alexandr\ V.\ Pukhlikov},
       DOI = {10.1017/CBO9780511721472.004},
       URL = {https://doi.org/10.1017/CBO9780511721472.004},
}

@article {HuberSchultzYe2023,
    AUTHOR = {Huber, Tim and Schultz, Daniel and Ye, Dongxi},
     TITLE = {Ramanujan-{S}ato series for {$1/\pi $}},
   JOURNAL = {Acta Arith.},
  FJOURNAL = {Acta Arithmetica},
    VOLUME = {207},
      YEAR = {2023},
    NUMBER = {2},
     PAGES = {121--160},
      ISSN = {0065-1036,1730-6264},
   MRCLASS = {11F03 (11F11)},
  MRNUMBER = {4564300},
MRREVIEWER = {Jia-Wei\ Guo},
       DOI = {10.4064/aa220621-19-12},
       URL = {https://doi.org/10.4064/aa220621-19-12},
}

@article {Fei,
    AUTHOR = {Fei, Jiarui},
     TITLE = {Mahler measure of 3{D} {L}andau-{G}inzburg potentials},
   JOURNAL = {Forum Math.},
  FJOURNAL = {Forum Mathematicum},
    VOLUME = {33},
      YEAR = {2021},
    NUMBER = {5},
     PAGES = {1369--1401},
      ISSN = {0933-7741,1435-5337},
   MRCLASS = {11R06 (11F67)},
  MRNUMBER = {4308625},
MRREVIEWER = {James\ McKee},
       DOI = {10.1515/forum-2020-0339},
       URL = {https://doi.org/10.1515/forum-2020-0339},
}

@book {Berndt1994,
    AUTHOR = {Berndt, Bruce C.},
     TITLE = {Ramanujan's notebooks. {P}art {IV}},
 PUBLISHER = {Springer-Verlag, New York},
      YEAR = {1994},
     PAGES = {xii+451},
      ISBN = {0-387-94109-6},
   MRCLASS = {11B68 (01A75 11F11 11N05)},
  MRNUMBER = {1261634},
MRREVIEWER = {Don\ Redmond},
       DOI = {10.1007/978-1-4612-0879-2},
       URL = {https://doi.org/10.1007/978-1-4612-0879-2},
}

@article {Cooper2012,
    AUTHOR = {Cooper, Shaun},
     TITLE = {Level 10 analogues of {R}amanujan's series for {$1/\pi$}},
   JOURNAL = {J. Ramanujan Math. Soc.},
  FJOURNAL = {Journal of the Ramanujan Mathematical Society},
    VOLUME = {27},
      YEAR = {2012},
    NUMBER = {1},
     PAGES = {59--76},
      ISSN = {0970-1249,2320-3110},
   MRCLASS = {11Yxx (11Fxx)},
  MRNUMBER = {2933486},
}

@book {Cooper2017,
    AUTHOR = {Cooper, Shaun},
     TITLE = {Ramanujan's theta functions},
 PUBLISHER = {Springer, Cham},
      YEAR = {2017},
     PAGES = {xviii+687},
      ISBN = {978-3-319-56171-4; 978-3-319-56172-1},
   MRCLASS = {11-02 (05A15 05A19 11F11 11F27 33-02)},
  MRNUMBER = {3675178},
MRREVIEWER = {Jeremy\ Lovejoy},
       DOI = {10.1007/978-3-319-56172-1},
       URL = {https://doi.org/10.1007/978-3-319-56172-1},
}

@incollection {Cooper2025,
    AUTHOR = {Cooper, Shaun},
     TITLE = {Ap\'ery-like sequences defined by four-term recurrence
              relations},
 BOOKTITLE = {Classical hypergeometric functions and generalizations},
    SERIES = {Contemp. Math.},
    VOLUME = {818},
     PAGES = {137--179},
 PUBLISHER = {Amer. Math. Soc., [Providence], RI},
      YEAR = {2025},
      ISBN = {978-1-4704-7653-3; [9781470479381]},
   MRCLASS = {11F11},
  MRNUMBER = {4901163},
MRREVIEWER = {Brundaban\ Sahu},
       DOI = {10.1090/conm/818/16373},
       URL = {https://doi.org/10.1090/conm/818/16373},
}

@article {CooperGeYe2015,
    AUTHOR = {Cooper, Shaun and Ge, Jinqi and Ye, Dongxi},
     TITLE = {Hypergeometric transformation formulas of degrees 3, 7, 11 and
              23},
   JOURNAL = {J. Math. Anal. Appl.},
  FJOURNAL = {Journal of Mathematical Analysis and Applications},
    VOLUME = {421},
      YEAR = {2015},
    NUMBER = {2},
     PAGES = {1358--1376},
      ISSN = {0022-247X,1096-0813},
   MRCLASS = {33C20 (33E05)},
  MRNUMBER = {3258323},
MRREVIEWER = {Christian\ Lavault},
       DOI = {10.1016/j.jmaa.2014.07.061},
       URL = {https://doi.org/10.1016/j.jmaa.2014.07.061},
}

\end{document}